\documentclass{amsart}

\usepackage[margin = 1in]{geometry}
\usepackage{amssymb}
\usepackage{amsmath}
\usepackage{amsthm}
\usepackage{hyperref}
\usepackage{orcidlink}

\renewcommand{\geq}{\geqslant}
\renewcommand{\leq}{\leqslant}

\newcommand{\sep}{\operatorname{sep}}

\newcommand{\R}{\mathbb{R}}
\newcommand{\Z}{\mathbb{Z}}

\newcommand{\C}{\mathbb{C}}

\renewcommand{\epsilon}{\varepsilon}

\newtheorem{theorem}{Theorem}[section]
\newtheorem{lemma}[theorem]{Lemma}
\newtheorem{proposition}[theorem]{Proposition}

\newtheorem{conjecture}[theorem]{Conjecture}

\theoremstyle{remark}
\newtheorem{definition}{Definition}

\title{Upper Bounds on Polynomial Root Separation}
\author{Greg Knapp}
\address{Department of Mathematics and Statistics\\University of Calgary\\Calgary, AB T2N 1N4\\ Canada\\\orcidlink{0000-0001-6808-5176}}
\email{greg.knapp@ucalgary.ca}
\author{Chi Hoi Yip}
\address{School of Mathematics\\ Georgia Institute of Technology\\Atlanta, GA 30332\\ United States\\\orcidlink{0000-0003-0753-1675}}
\email{cyip30@gatech.edu}


\begin{document}

\begin{abstract}
In this paper, we consider the relationship between the Mahler measure of a polynomial and its separation.  In 1964, Mahler proved that if $f(x) \in \mathbb{Z}[x]$ is separable of degree $n$, then $\operatorname{sep}(f) \gg_n M(f)^{-(n-1)}$.  This spurred further investigations into the implicit constant involved in that relation, and it led to questions about the optimal exponent on $M(f)$ in that relation.  However, there has been relatively little study concerning upper bounds on $\operatorname{sep}(f)$ in terms of $M(f)$.  In this paper, we prove that if $f(x) \in \mathbb{C}[x]$ has degree $n$, then $\operatorname{sep}(f) \ll n^{-1/2}M(f)^{1/(n-1)}$. Moreover, this bound is sharp up to the implied constant factor.  We further investigate the constant factor under various additional assumptions on $f(x)$, for example, if it only has real roots.
\end{abstract}

\maketitle

\section{Introduction}

A well-studied subject in number theory is the distribution of roots of polynomials.  This tradition includes, for example, Descartes' Rule of Signs, the Gauss-Lucas Theorem, and the Schinzel-Zassenhaus Conjecture.  Here, we study the relationship between the separation of a polynomial (the minimal distance between its roots) and the Mahler measure of that same polynomial.  To be precise, we define these terms here.

\begin{definition}\label{def:separation}
	Given a polynomial $f(x) \in \C[x]$ with roots $\alpha_1,\dots,\alpha_n \in \C$, the \underline{separation} of $f(x)$ is the quantity \[\sep(f) := \min_{\alpha_i \neq \alpha_j} |\alpha_i - \alpha_j|.\]
\end{definition}

Since, for any nonzero $b \in \C$, $\sep(f) = \sep(bf)$, we assume in the remainder of the paper that all polynomials are monic.  Our results can be easily adapted otherwise.

\begin{definition}\label{def:mahlerMeasure}
	Given a monic polynomial $f(x) \in \C[x]$ with roots $\alpha_1,\dots,\alpha_n \in \C$, the \underline{Mahler measure} of $f(x)$ is the quantity \[M(f) :=\prod_{i=1}^n \max\{1,|\alpha_i|\}.\]
\end{definition}

These two quantities are useful for rather different reasons.  Knowledge about root separation is helpful for writing algorithms that compute the number of real roots of a polynomial (see Koiran's work in \cite{Koiran2019}), for bounding the number of solutions to Thue equations (see Grundman and Wisniewski's paper \cite{Grundman2013}), and for deriving lower bounds on the absolute values of certain products of algebraic integers (see Albayrak, Ghosh, Nguyen, and the first author's upcoming paper \cite{Albayrak2024}).

Meanwhile, Mahler measure is a versatile tool for measuring the complexity of a polynomial. On the one hand, Mahler measure contains information about the roots of a polynomial via its definition.  On the other hand, Mahler measure contains information about the coefficients of the polynomial: for any polynomial $f(x) \in \C[x]$ of degree $n$, $M(f) \asymp_n H(f)$ where $H(f)$ is the maximum absolute value of its coefficients \cite[Lemma 1.6.7]{BombieriEnrico2001HiDG}.  Additionally, Mahler measure preserves algebraic information about polynomials since it is multiplicative.  These facts together make Mahler measure valuable for translating information about the roots of a polynomial into information about its coefficients, and vice versa.  Mahler measure is the main object in Lehmer's Conjecture\footnote{See \cite{Lehmer1933} for Lehmer's statement of the problem, though he does not actually conjecture an answer.}, and the authors invite the reader to investigate Smyth's excellent survey \cite{Smyth2008} for information on the use and study of the Mahler measure.

In general, it is valuable to find effective lower bounds on polynomial root separation.  After all, separation gives the minimum distance between distinct roots of a polynomial, so a lower bound on separation gives a lower bound on the distance between any two roots of that polynomial.  Indeed, this is how separation is used in \cite{Albayrak2024,Grundman2013,Koiran2019}.  A foundational lower bound on separation is the following result of Mahler \cite{Mahler1964}.  It relies on the discriminant, defined for monic $f(x) \in \C[x]$ with roots $\alpha_1,\dots,\alpha_n$ to be \[D(f) := \prod_{1 \leq i < j \leq n} (\alpha_i - \alpha_j)^2.\]

\begin{theorem}\label{thm:mahlerLowerBound}
 Let $f(x) \in \C[x]$ be separable of degree $n \geq 2$ with discriminant $D(f)$.  Then
	\begin{align*}
		\sep(f) > \frac{\sqrt{3|D(f)|}}{n^{(n+2)/2}M(f)^{n-1}};
	\end{align*}
 in particular, if $f(x) \in \Z[x]$, then 
    \begin{align*}
		\sep(f) > \frac{\sqrt{3}}{n^{(n+2)/2}M(f)^{n-1}}.
	\end{align*}
\end{theorem}

This particular result has spurred much investigation, as it leaves many questions open.  Is the exponent of $n-1$ on $M(f)$ optimal?  This remains unknown, though Bugeaud and Dujella show that the optimal exponent is at least $(2n-1)/3$ in \cite{Bugeaud2014} for $f \in \Z[x]$, while Evertse considers other ways to improve the exponent in \cite{Evertse2004}.  Is the separability hypothesis necessary?  Rump removes the separability hypothesis in \cite{Rump1979}, and Dujella and Pejkovi\'c improve his result in \cite{Dujella2017}.  Can better results be obtained with additional hypotheses (e.g. if $f(x)$ is assumed to be irreducible and/or monic)?  Many papers have looked at these kinds of questions, including \cite{Bugeaud2011,Dujella2017,Koiran2019}.  What if we look at distances between the absolute values of the roots, rather than distances between the roots themselves?  Bugeaud, Dujella, Fang, Pejkovi\'c and Salvy look at these questions in \cite{Bugeaud2022a,Bugeaud2017}.

While lower bounds on separation have been well-studied, the authors are not aware of any attempt to provide upper bounds on separation.  On the face of it, an upper bound on separation provides less intuitive information about a polynomial's roots than does a lower bound.  However, upper bounds on separation can provide us with useful information for Lehmer's conjecture, which we will describe after our main theorem.

One can easily obtain a trivial upper bound on separation using the triangle inequality: $\sep(f) = \min_{\alpha_i \neq \alpha_j} |\alpha_i - \alpha_j| \leq 2M(f).$  Combining Mahler's inequality $|D(f)| \leq n^nM(f)^{2n-2}$ \cite[Theorem 1]{Mahler1964} with the fact that $\sep(f)^{n(n-1)} \leq |D(f)|$ when $f(x)$ is separable gives the improved upper bound
\begin{align*}
	\sep(f) \leq n^{1/(n-1)}M(f)^{2/n}
\end{align*}
for separable polynomials $f(x) \in \C[x]$.  However, numerical experiments conducted by the first author at the outset of this project in \cite{Knapp2023a} suggested that this upper bound is still sub-optimal, and indeed, we can prove the following result.

\begin{theorem}\label{thm:optimalExponent}
    Let $f(x) \in \C[x]$ be monic and separable of degree $n \geq 2$.  Then \[\sep(f) \leq \min\left\{2,\frac{34}{\sqrt{n}}\right\}M(f)^{1/(n-1)}.\]  If, in addition, $f(x) \in \R[x]$ and one of the roots of $f(x)$ with minimum absolute value is not real (for example, if $f(x) \in \R[x]$ has no real roots), then \[\sep(f) \leq \min\left\{2,\frac{34}{\sqrt{n}}\right\} M(f)^{1/n}.\]
\end{theorem}

In both cases, the dependence of the bound on $n$ is optimal, though the constant 34 could possibly be improved.  For some particular signatures\footnote{A polynomial $f(x) \in \R[x]$ has signature $(t,s)$ if it has $t$ real roots and $s$ pairs of complex conjugate roots.} of $f(x) \in \R[x]$, we are able to find the optimal coefficient as we describe in the following theorem.

\begin{theorem}\label{thm:improvedConstant}
Suppose $f(x) \in \R[x]$ is monic and separable of degree $n \geq 2$ and signature $(t,s)$ where $s = 0$, or $(t,s) = (1,1)$, or $(t,s) = (0,2)$. Then \[\sep(f) \leq C_{t,s}M(f)^{1/(n - \delta)}\] where $\delta$ is $1$ if $t \neq 0$ and $0$ otherwise, and $C_{t,s}$ is defined by the following table:\\
	\begin{center}
		\begin{tabular}{| c || c | c | c |}\hline
			$(t,s)$ & $(1,1)$ & $(0,2)$ & $(t,0)$\\\hline
			$C_{t,s}$ & $\sqrt{3}$ & $\sqrt{2}$ & $(2e+o(1))/t$\\\hline
		\end{tabular}
	\end{center}
\end{theorem}

It is worth noting that each of our main theorems produce the same bounds on absolute separation, defined in \cite{Bugeaud2022a} to be \[\operatorname{abs\, sep}(f) = \min_{|\alpha_i| \neq |\alpha_j|} ||\alpha_i| - |\alpha_j||,\] since we have $\operatorname{abs\, sep}(f) \leq \sep(f)$ for all $f$.

With our main theorems described, we can detail the connection to Lehmer's Conjecture.

\begin{conjecture}[Lehmer's Conjecture]
	There exists an absolute constant $\mu > 1$ so that for every irreducible, noncyclotomic $f(x) \in \Z[x]$ of degree at least 2, $M(f) \geq \mu$.
\end{conjecture}

Observe that Theorem \ref{thm:mahlerLowerBound} provides a lower bound on $M(f)$ which is useful when $\sep(f)$ is small: \[M(f) \geq \left(\frac{\sqrt{3}}{n^{(n+2)/2}\sep(f)}\right)^{1/(n-1)}.\]  On the other hand, Theorem \ref{thm:optimalExponent} provides a lower bound on $M(f)$ that is useful when $\sep(f)$ is large: \[M(f) \geq \left(\frac{\sep(f)}{\min\{2,34/\sqrt{n}\}}\right)^{n-1}.\]

Hence, proving Lehmer's Conjecture reduces the case where $f(x)$ satisfies \[\frac{\sqrt{3}}{n^{(n+2)/2}\mu^{n-1}} < \sep(f) < \min\left\{2,\frac{34}{\sqrt{n}}\right\}\mu^{1/(n-1)}.\]

%%%%%%%%%%%%%%%%%%%%%%%%%%%%%%%%%%%%%%%%%%%%%%%%%%%%%%%%%%%%%%%%%%%%%%%
%%%%%%%%%%%% THE OPTIMAL UPPER BOUND %%%%%%%%%%%%%%%%%%%%%%%%%%%%%%%%%%
%%%%%%%%%%%%%%%%%%%%%%%%%%%%%%%%%%%%%%%%%%%%%%%%%%%%%%%%%%%%%%%%%%%%%%%

\section{The Optimal upper bound up to constant factor}

The main goal of this section is to prove Theorem \ref{thm:optimalExponent}.

\begin{proof}[Proof of Theorem~\ref{thm:optimalExponent}]
    Write $f(x) = \prod_{i=1}^n (x - \alpha_i)$ where $|\alpha_1| \leq |\alpha_2| \leq \cdots \leq |\alpha_n|$.  
    
    We first prove that $\sep(f) \leq 2M(f)^{1/(n-1)}$.  Note that for any $j \geq 2$, we have $0 < |\alpha_j - \alpha_1| \leq 2|\alpha_j|$.  As a consequence,
    \begin{align*}
        \sep(f)^{n-1} \leq |\alpha_2 - \alpha_1||\alpha_3 - \alpha_1|\cdots|\alpha_n - \alpha_1| \leq \prod_{j = 2}^n 2|\alpha_j| \leq 2^{n-1}M(f)
    \end{align*}
    and the claim follows.
    
    Next, we show that if $f(x) \in \R[x]$ and if there exists $\alpha_j \notin \R$ with $|\alpha_j| = |\alpha_1|$, then actually $\sep(f) \leq 2M(f)^{1/n}$.  The argument is very similar, except we may now assume that $\alpha_1$ and $\alpha_2$ are complex conjugates.  In this case,
    \begin{align*}
        \sep(f)^n \leq |\alpha_2 - \alpha_1|^2|\alpha_3 - \alpha_1|\cdots|\alpha_n - \alpha_1| \leq \prod_{j=1}^n 2|\alpha_j| \leq 2^nM(f)
    \end{align*}
    and the claim again follows.

    Now we return to the situation where we only assume $f(x) \in \C[x]$ is separable of degree $n$, and we aim to show that \[\sep(f) \leq \frac{34}{\sqrt{n}}\cdot M(f)^{1/(n-1)}.\]
    Since we have $\frac{34}{\sqrt{n}} \geq 2$ for $n \leq 289$, we assume that $n \geq 290$ in the following discussion.

    Let $r = \sep(f)/2$ and for any $R \geq 0$, let \[N(R) := \#\{ 1\leq i \leq n: |\alpha_i| < R\}.\]  Additionally, for any $z \in \C$, let $B_R(z)$ denote the open ball of radius $R$ centered at $z$.  We may safely assume that $r \geq \frac{1}{\sqrt{n}}$, as the theorem easily follows otherwise.

    Next, observe that for any $i \neq j$, $B_r(\alpha_i)$ and $B_r(\alpha_j)$ are disjoint.  Additionally, for any $R > 0$, \[B_{R+r}(0) \supsetneq \bigcup_{|\alpha_i| \leq R} B_r(\alpha_i).\]  As a consequence, \[\pi(R+r)^2 = \operatorname{vol}(B_{R+r}(0)) > \sum_{|\alpha_i| \leq R} \operatorname{vol}(B_r(\alpha_i)) = N(R)\cdot\pi r^2,\] implying that 
    \begin{align}\label{eq:numSmallRoots}
        N(R) < \left(\frac{R}{r} + 1\right)^2.
    \end{align}
    Now, for each positive integer $j \leq \lceil \log_2(n)\rceil - 1 =: L$ (here we use the assumption that $n \geq 3$ to obtain $L \geq 1$), define $R_j = r(\sqrt{n/2^j} - 1)$.  Note that $R_j > 0$ for $j \leq L$.  For each $j$, we aim to bound from below the number of roots $\alpha_i$ which satisfy $R_j < |\alpha_i|$.

    Observe that $N(R_j) < n/2^j$ by inequality~\eqref{eq:numSmallRoots}, and hence, for each $1 \leq j \leq L$, there must be at least $n(2^j - 1)/2^j$ roots $\alpha_i$ which satisfy $|\alpha_i| \geq R_j$.  Therefore,
    \begin{align}\label{eq:initMahlerLower2}
        M(f)    &\geq \prod_{|\alpha_i| \geq R_L} |\alpha_i|= (r\sqrt{n})^{n-N(R_L)} \prod_{|\alpha_i| \geq R_L} \frac{|\alpha_i|}{r\sqrt{n}}\nonumber\\
                &\geq (r\sqrt{n})^{n-N(R_L)}\left(\frac{R_1}{r\sqrt{n}}\right)^{n/2}\left(\frac{R_2}{r\sqrt{n}}\right)^{n/4}\cdots \left(\frac{R_L}{r\sqrt{n}}\right)^{n/2^L}.
    \end{align}
The remainder of this proof will be dedicated to showing that \[\left(\frac{R_1}{r\sqrt{n}}\right)^{n/2}\left(\frac{R_2}{r\sqrt{n}}\right)^{n/4}\cdots \left(\frac{R_L}{r\sqrt{n}}\right)^{n/2^L} \geq \frac{1}{4n^2\cdot16^n}.\]

    First, note that for any $j \leq L$, 
    \begin{equation*}
    \frac{R_j}{r\sqrt{n}} = 2^{-j/2} - n^{-1/2} = \frac{2^{-j} - n^{-1}}{2^{-j/2} + n^{-1/2}} = \frac{n - 2^j}{n\cdot2^j(2^{-j/2} + n^{-1/2})}.  
    \end{equation*}
    Now, since $j \leq L$, we have that $n^{-1/2} \leq 2^{-j/2}$, implying that $2^{-j/2} + n^{-1/2} \leq 2^{-j/2 + 1}$.  Hence, \[\frac{R_j}{r\sqrt{n}} \geq \frac{n - 2^j}{n\cdot2^{j/2 + 1}}.\]
    We now make slightly different estimates for $j = L$ and for $j < L$.  If $j = L$, we have \[\frac{R_L}{r\sqrt{n}} \geq \frac{1}{n\cdot2^{L/2 + 1}}.\]     If $j < L$, we use the fact that $2^L < n \leq 2^{L+1}$ to find that \[\frac{R_j}{r\sqrt{n}} \geq \frac{n - 2^j}{n\cdot2^{j/2 + 1}} \geq \frac{2^L - 2^j}{2^{L + j/2 + 2}} = \frac{1 - 2^{j-L}}{2^{j/2 + 2}} \geq \frac{1}{2^{j/2 + 3}}.\]
    Consequently, we can write
    \begin{align*}
        \prod_{j=1}^L \left(\frac{R_j}{r\sqrt{n}}\right)^{n/2^j} &\geq \left[\left(\frac{1}{n\cdot2^{L/2+1}}\right)^{2^{-L}}\cdot\prod_{j=1}^{L-1}2^{(-j/2 - 3)/2^j}\right]^n\\
        &= \left[\frac{1}{(2n)^{2^{-L}}} \cdot 2^{-\sum_{j=1}^L j/2^{j+1}} \cdot 8^{-\sum_{j=1}^{L-1} 2^{-j}}\right]^n\\
        & \geq \left[\frac{1}{(2n)^{2^{-L}}} \cdot 2^{-\sum_{j\geq1} j/2^{j+1}} \cdot 8^{-\sum_{j\geq1} 2^{-j}}\right]^n.
    \end{align*}
    The series in the above exponents converge, with \[\sum_{j \geq 1} 2^{-j} = 1 \quad \text{and} \quad \sum_{j \geq 1} j \cdot 2^{-(j+1)} = 1.\]  Hence, since $n \leq 2^{L+1}$, \[\prod_{j=1}^L \left(\frac{R_j}{r\sqrt{n}}\right)^{n/2^j} \geq \frac{1}{(2n)^{n/2^L}}\cdot(2^{-1}\cdot 8^{-1})^n \geq \frac{1}{4n^2\cdot16^n}.\]

    Finally, we can return to inequality~\eqref{eq:initMahlerLower2} and we acquire 
    \begin{equation}\label{eq:eq3}
    M(f) \geq \frac{1}{4n^2\cdot16^n}\cdot(r\sqrt{n})^{n - N(R_L)}.      
    \end{equation}
    Since $N(R_L) < n/2^L \leq 2$ and since $N(R_L)$ is an integer, we actually have $N(R_L) \leq 1$.  Since we have assumed that $r \geq 1/\sqrt{n}$, inequality~\eqref{eq:eq3} implies \[M(f) \geq \frac{1}{4n^2\cdot16^n}\cdot(r\sqrt{n})^{n-1}\] and we get \[\sep(f) = 2r \leq \frac{2\cdot(4n^2)^{1/(n-1)}\cdot16^{1 + 1/(n-1)}}{\sqrt{n}}\cdot M(f)^{1/(n-1)} \leq \frac{34}{\sqrt{n}} \cdot M(f)^{1/(n-1)}, \] where we use the fact that $n \geq 290$ in the final inequality.

    Note that the additional, stronger statement for $f(x) \in \R[x]$ also follows.  If $f(x)$ has a nonreal root among its roots of minimum absolute value, then we cannot have $N(R_L) = 1$, so we must have $N(R_L) = 0$. Hence, inequality~\eqref{eq:eq3} implies that\[M(f) \geq \frac{1}{4n^2}\left(\frac{r\sqrt{n}}{16}\right)^n,\] and the theorem statement again follows.
\end{proof}

We now demonstrate that both bounds are optimal except possibly the constant 34.  We do this by constructing a family of polynomials whose separation nearly attains the upper bound which we have now proven.

Before we do this, we note the following fact about the Gaussian integers.  For $r \geq 0$, let $G(r)$ denote the number of Gaussian integers $a + bi$ for $a,b \in \Z$ which satisfy $|a + bi| \leq r$.  A well-known result of Gauss (see for example \cite[page 101]{Berndt2018}) states that for $r \geq \sqrt{2}$ we have \[\pi(r - \sqrt{2})^2 \leq G(r) \leq \pi (r + \sqrt{2})^2.\] 

Fix an integer $n \geq 2$ and a real number $t \geq 1$.  Set $R = \sqrt{n/\pi} + \sqrt{2}$, which implies that there are at least $n$ Gaussian integers $a + bi$ with $|a + bi| \leq R$. Select $n$ distinct Gaussian integers $\alpha_1,\dots,\alpha_n$ so that
\begin{enumerate}
    \item $\alpha_1 = 0$.
    \item For each $1 \leq j \leq n$, we have $|\alpha_j| \leq R$.
\end{enumerate}
Now define \[f_t(x) = \prod_{j=1}^n (x - t\alpha_j).\]  For $j \geq 2$, we have $|\alpha_j| \geq 1$, implying that \[M(f_t) = \prod_{j=2}^n t|\alpha_j| \leq (tR)^{n-1} \leq (1.6\sep(f_t)\sqrt{n})^{n-1},\] implying that $\sep(f_t) \geq \frac{1}{1.6\sqrt{n}}M(f_t)^{1/(n-1)}.$  Hence, the first upper bound on separation in Theorem \ref{thm:optimalExponent} has the best possible dependence on $n$.

For the remaining examples, we take $R = \sqrt{(n+1)/\pi} + \sqrt{2}$ and we choose $n$ distinct Gaussian integers $\beta_1,\dots,\beta_n$ so that
\begin{enumerate}
    \item \label{it:sizeRestriction} For each $1 \leq j \leq n$, we have $0 < |\beta_j| \leq R$.
    \item The set $\{\beta_1,\dots,\beta_n\}$ is closed under complex conjugation.
    \item There exists $\beta_j \notin \R$ so that $|\beta_j| = \min_k|\beta_k|$.
\end{enumerate}

Observe that condition \ref{it:sizeRestriction} together with the fact that the $\beta_k$ are Gaussian integers implies that $\min_k |\beta_k| \geq 1$.  For real $t \geq 1$, we define \[g_t(x) = \prod_{j=1}^n (x - t\beta_j) \in \R[x]\] and now we have \[M(g_t) = \prod_{j=1}^n t|\beta_j| \leq (tR)^n \leq (1.7\sqrt{n}\sep(g_t))^n,\] which implies that $\sep(g_t) \geq \frac{1}{1.7\sqrt{n}}M(g_t)^{1/n}$.  Hence, the second upper bound on separation in Theorem \ref{thm:optimalExponent} has the best possible dependence on $n$.

We remark that the constant factors of 1.6 and 1.7 can be improved by replacing the square lattice of Gaussian integers with the hexagonal lattice, and with some extra work using partial summation to find a more efficient upper bound on $M(f_t)$ and $M(g_t)$.

%%%%%%%%%%%%%%%%%%%%%%%%%%%%%%%%%%%%%%%%%%%%%%%%%%%%%%%%%%%%%%%%%%%%%%%
%%%%%%%%%%%% IMPROVING THE CONSTANT FACTOR %%%%%%%%%%%%%%%%%%%%%%%%%%%%
%%%%%%%%%%%%%%%%%%%%%%%%%%%%%%%%%%%%%%%%%%%%%%%%%%%%%%%%%%%%%%%%%%%%%%%
\section{Improving the Constant Factor}

Our main goal of this section is to prove Theorem \ref{thm:improvedConstant}, which we will do in pieces.

\begin{proposition}\label{totallyRealSeparationProp}
	Let $f(x) \in \R[x]$ be monic and separable with degree $n \geq 4$.  Suppose further that all $n$ of the roots of $f$ are real.  Then \[\sep(f) \leq \frac{6.33}{n}\cdot M(f)^{1/(n-1)}.\]
 Moreover, as $n \to \infty$, we have 
 \begin{equation}\label{eq:asymp}
\sep(f) \leq \frac{2e+o(1)}{n}\cdot M(f)^{1/(n-1)},     
 \end{equation}
 and the constant $2e$ is asymptotically sharp.
\end{proposition}

\begin{proposition}\label{cubicSeparationProp}
	Let $f(x) \in \R[x]$ be monic and separable with degree $3$.  If $f(x)$ has exactly one real root, then \[\sep(f) \leq \sqrt{3M(f)}.\]  Moreover, the constant $\sqrt{3}$ is optimal.
\end{proposition}

\begin{proposition}\label{quarticSig02Prop}
	Let $f(x) \in \R[x]$ be monic and separable with degree $4$. If $f(x)$ has no real roots, then \[\sep(f) \leq \sqrt{2} \cdot M(f)^{1/4}.\]  Moreover, the constant $\sqrt{2}$ is optimal.
\end{proposition}

These propositions cover each of the cases stated in Theorem \ref{thm:improvedConstant} except the case where the signature is $(2,0)$ or $(3,0)$, in which case the result follows from Theorem \ref{thm:optimalExponent}.

We first prove a lemma on bounding binomial coefficients.

\begin{lemma}\label{lem:binom}
    For any positive integer $n \geq 3$, \[\binom{n}{\lfloor n/2\rfloor} \leq \frac{2^{n+1}}{\sqrt{\pi(2n+1)}}.\]
\end{lemma}

\begin{proof}
    It is known that for any positive integer $\ell$, \[\binom{2\ell}{\ell} \leq \frac{2^{2\ell}}{\sqrt{\pi (\ell+1/4)}},\] see for example \cite{MO}. Now, if $n$ is even, we have \[\binom{n}{\lfloor n/2\rfloor} \leq \frac{2^{n}}{\sqrt{\pi(n/2 + 1/4)}} = \frac{2^{n+1}}{\sqrt{\pi(2n + 1)}}.\] If $n$ is odd, say $n = 2k + 1$, then
    \begin{align*}
        \binom{n}{\lfloor n/2\rfloor}   = \binom{2k + 1}{k}
                                        = \frac{2k + 1}{k + 1} \binom{2k}{k}
                                        \leq \frac{2k + 1}{k + 1}\cdot\frac{2^{2k}}{\sqrt{\pi(k+1/4)}}
                                        = \frac{n}{n + 1}\cdot\frac{2^{n+1}}{\sqrt{\pi (2n-1)}}
    \end{align*}
    and one can now check that $\frac{n}{(n+1)\sqrt{2n-1}} \leq \frac{1}{\sqrt{2n + 1}}$ to complete the proof.
\end{proof}

Next we prove Proposition~\ref{totallyRealSeparationProp}.

\begin{proof}[Proof of Proposition~\ref{totallyRealSeparationProp}]
Denote the closest root of $f(x)$ to 0 by $\alpha$ and let $r = \sep(f)$.  Suppose that there are $s$ roots of $f(x)$ which are less than $\alpha$ and $t$ roots of $f(x)$ which are greater than $\alpha$.  Define \[g(x) = \prod_{i = 1}^t (x - \alpha - ri) \cdot (x - \alpha) \cdot \prod_{j = 1}^s (x - \alpha + rj)\] and observe that we have $M(f) \geq M(g)$ because the roots of $g$ are no further from the origin than the corresponding roots of $f$.  Furthermore, we have $\sep(f) = r = \sep(g)$, so it suffices to prove the proposition for $g(x)$.

We first consider the case that all roots of $g$ have the same sign. In this case, we can assume $\alpha\geq 0$ without loss of generality. Then the roots of $g$ are simply given by $\alpha, \alpha+r, \cdots, \alpha+(n-1)r$. Thus, $$M(g)\geq \prod_{j=1}^{n-1} (\alpha+jr) \geq \prod_{j=1}^{n-1} (jr)=(n-1)!r^{n-1}.$$ It follows that $$\sep(g)=r\leq  \left(\frac{M(g)}{(n-1)!}\right)^{1/(n-1)}<\frac{e}{n} \cdot M(g)^{1/(n-1)}.$$

Next we consider the case where not all roots of $g$ have the same sign. Let $\beta$ be the closest root of $g(x)$ to the origin and write \[g(x) = \prod_{i = 1}^T (x - \beta - ri) \cdot (x - \beta) \cdot \prod_{j = 1}^S (x - \beta + rj).\]  Since $\beta$ is the closest root of $g(x)$ to 0 and since $g(x)$ has a root with the opposite sign of $\beta$, we find that $|\beta| \leq r/2$.  We may also assume without loss of generality that $\beta \leq 0$ (else, we may apply the same proof to $g(-x)$).  We now have
	\begin{align}
		M(g)	&\geq \prod_{i = 1}^{T} |\beta + ri| \cdot \prod_{j = 1}^{S} |\beta - rj| \geq \prod_{i = 1}^{T} \left(ri - \frac{r}{2}\right) \cdot \prod_{j = 1}^{S} \left(rj + \frac{r}{2}\right)= r^{n-1} \cdot \prod_{i = 1}^{T} \left(i - \frac{1}{2}\right) \cdot \prod_{j = 1}^{S} \left(j + \frac{1}{2}\right).\label{eq:mOfGLowerBound}
	\end{align}

	Getting a handle on the lower bound given in inequality~\eqref{eq:mOfGLowerBound} will require some use of the gamma function.  We use the fact that $\Gamma\left(\frac{1}{2}\right) = \sqrt{\pi}$ together with the usual fact that $\Gamma(z) = (z - 1)\Gamma(z-1)$ to note that \[\prod_{i = 1}^T \left(i - \frac{1}{2}\right) = \frac{\Gamma\left(T + \frac{1}{2}\right)}{\sqrt{\pi}}\qquad \text{and}\qquad \prod_{j = 1}^S \left(j + \frac{1}{2}\right) = \frac{2\Gamma\left(S + \frac{3}{2}\right)}{\sqrt{\pi}}.\] 
By Wendel's inequality \cite{W48}, if $m$ is a nonnegative integer, we have
 $$
 \Gamma\bigg(m+\frac{1}{2}\bigg)> \frac{m\Gamma(m)}{\sqrt{m+\frac{1}{2}}}=\frac{m!}{\sqrt{m+\frac{1}{2}}}.
 $$
Thus, 
	\begin{align*}
		\prod_{i = 1}^{T} \left(i - \frac{1}{2}\right) = \frac{\Gamma\left(T + \frac{1}{2}\right)}{\sqrt{\pi}} > \frac{T!}{\sqrt{\pi(T+\frac{1}{2})}}, \quad \quad		\prod_{j = 1}^{S} \left(j + \frac{1}{2}\right) = \frac{2\Gamma\left(S + \frac{3}{2}\right)}{\sqrt{\pi}} > \frac{2(S +1)!}{\sqrt{\pi(S + \frac{3}{2})}}.
	\end{align*}

	Now we have 
	\begin{align}\label{eq:mOfGLowerFactorial}
		M(g) \geq r^{n-1} \cdot \frac{T!}{\sqrt{\pi(T+\frac{1}{2})}}\cdot \frac{2(S +1)!}{\sqrt{\pi(S + \frac{3}{2})}}
	\end{align} 
	from inequality \eqref{eq:mOfGLowerBound} and we aim to estimate the right-hand side of this inequality from below in terms of $n$.  We have the restrictions $0 \leq S,T \leq n$ and $S + T + 1 = n$, so we can replace $S + 1$ in equation \eqref{eq:mOfGLowerFactorial} by $n - T$ to find 
	\begin{align}
		\frac{T!}{\sqrt{\pi(T+\frac{1}{2})}}\cdot \frac{2(S +1)!}{\sqrt{\pi(S + \frac{3}{2})}}
		&= \frac{T!}{\sqrt{\pi(T+\frac{1}{2})}}\cdot \frac{2(n-T )!}{\sqrt{\pi(n-T + \frac{1}{2})}} \notag\\
		&= \frac{2n!}{\pi\binom{n}{T}\sqrt{(T + \frac{1}{2})(n - T + \frac{1}{2})}} \geq \frac{4n!}{\pi\binom{n}{\lfloor n/2 \rfloor} (n+1)} \label{eq:binomial}
	\end{align}

	Now, using Lemma~\ref{lem:binom} and Robbins' bound $n!>\sqrt{2\pi n} (n/e)^n e^{1/(12n+1)}$ \cite{R55} gives 
	\begin{align*}
	\frac{T!}{\sqrt{\pi(T+\frac{1}{2})}}\cdot \frac{2(S +1)!}{\sqrt{\pi(S + \frac{3}{2})}} 
		&\geq \frac{4n!}{\pi\binom{n}{\lfloor n/2 \rfloor} (n+1)}
  \geq \frac{n!\cdot \sqrt{2n+1}}{\sqrt{\pi}2^{n-1} (n+1)}\\
		&> \frac{\sqrt{2n}(n/e)^n e^{\frac{1}{12n + 1}}\sqrt{2n + 1}}{2^{n-1}(n + 1)} 
	=\left(\frac{n}{2e}\right)^n \cdot \frac{2\sqrt{2n(2n+1)}e^{\frac{1}{12n + 1}}}{n+1} 
		> 3.46\left(\frac{n}{2e}\right)^n
	\end{align*}
	by the fact that $n \geq 4$. 

	Finally, we return to inequality~\eqref{eq:mOfGLowerFactorial} to find that that 
	\begin{align*}
		M(g) &\geq 3.46\sep(g)^{n-1} \left(\frac{n}{2e}\right)^n,
	\end{align*}
	and we can conclude that \[\sep(g) \leq \left(\frac{M(g)}{3.46}\right)^{\frac{1}{n-1}} \left(\frac{2e}{n}\right)^{\frac{n}{n-1}} \leq \frac{2e(2e/3.46)^{1/3}}{n^{\frac{n}{n-1}}} M(g)^{\frac{1}{n-1}} \leq \frac{6.33}{n} M(g)^{\frac{1}{n-1}}.\]

 Asymptotically, using Stirling's approximation $n! \sim \sqrt{2\pi n} (n/e)^n$, we can combine inequality~\eqref{eq:mOfGLowerFactorial} and inequality~\eqref{eq:binomial} to deduce that
 $$
 M(g) \geq \sep(g)^{n-1} \frac{4n!}{\pi\binom{n}{\lfloor n/2 \rfloor} (n+1)} \geq (4-o(1))\sep(g)^{n-1} \left(\frac{n}{2e}\right)^n,
 $$
 as $n \to \infty$. It follows that $\sep(g)\leq \frac{2e+o(1)}{n}  M(g)^{\frac{1}{n-1}}$ as $n \to \infty$. 
 
Finally, we show that the constant $2e$ is asymptotically sharp in inequality~\eqref{eq:asymp}. Let $r>1$ be fixed. For each integer $n \geq 4$, set $f_n(x)=\prod_{j=-m}^{m} (x-jr)$ if $n=2m+1$ is odd, and $f_n(x)=\prod_{j=-m}^{m+1} (x-jr)$ if $n=2m+2$ is even. Stirling's approximation then yields $\sep(f_n)=\frac{2e+o(1)}{n}  M(f_n)^{\frac{1}{n-1}}$, as $n \to \infty$. 
\end{proof}

Next, we consider cubic polynomials which have only one real root.

\begin{proof}[Proof of Proposition \ref{cubicSeparationProp}.] 
   Suppose that the real root of $f(X)$ is $\alpha$ and without loss of generality, we may assume that $\alpha \geq 0$.  Let $\beta$ denote the complex root of $f(X)$ with positive imaginary part.  Write $\beta = x + iy$ for $x,y \in \R$.

	We claim that we may assume that $x \leq 0$.  If not, then set $\beta' = -x + iy$ and $g(X) = (X - \alpha)(X - \beta')(X - \overline{\beta'})$.  Since $\sep(g) \geq \sep(f)$ and $M(g) = M(f)$, proving the proposition for $g(X)$ will prove it for $f(X)$.  Hence, we only need to prove the proposition under the assumption that $x \leq 0$.

	Let $R = |\beta|$.  Note that if $y \leq \frac{\sqrt{3}R}{2}$, then \[\sep(f) \leq |\beta - \overline{\beta}| = 2y \leq \sqrt{3}R \leq \sqrt{3}M(f)^{1/2}\] and we are done.  Hence, for the rest of the proof, assume that $y > \frac{\sqrt{3}R}{2}$.  In particular, this implies that $y > -x\sqrt{3}$.

    Our next goal is to reduce to the case where $\alpha$ is ``small.''  Let $t$ be the unique value in $\R$ for which the points $t,\beta,\overline{\beta}$ form an equilateral triangle.  One can check that $t = \sqrt{3}y + x$.  Note that $t > -x\sqrt{3} + x = -x(\sqrt{3}-1) \geq 0$.  Now, if $\alpha \geq t$, set $h(X) = (X - t)(X - \beta)(X - \overline{\beta})$; we have $\sep(f) = 2y = \sep(h)$ and $M(f) \geq M(h)$, so it suffices to prove the proposition for $h$.  Hence, we may assume that $0 \leq \alpha < t$. 
    
    Note that $M(f)=\max \{1, \alpha\} \cdot \max \{1, x^2+y^2\}$ and $\sep(f)=|\alpha-\beta|=\sqrt{(\alpha-x)^2+y^2}$. Our goal is to show $M(f)/\sep(f)^2 \geq \frac{1}{3}$. We divide our proof into two cases:

    \noindent\underline{Case 1:} $x^2 + y^2 \geq 1$.

    Within this case, we have two subcases.  First, assume that $\alpha < 1$.  Then we have \[\frac{M(f)}{\sep(f)^2} = \frac{x^2 + y^2}{(\alpha - x)^2 + y^2} \geq \frac{x^2 + y^2}{(1 - x)^2 + y^2}.\]  We aim to show that \[\frac{x^2 + y^2}{(1 - x)^2 + y^2} \geq \frac{1}{3}\] or, equivalently, that $2(x^2 + y^2) \geq 1 - 2x.$  If $-1/2 \leq x \leq 0$, then this follows immediately from the fact that $x^2 + y^2 \geq 1$.  If $x < -1/2$, then we can use the fact that $y > -x\sqrt{3}$ to deduce that $2(x^2 + y^2) \geq 8x^2 > 1 - 2x.$ % Hence, \[\frac{x^2 + y^2}{(1 - x)^2 + y^2} \geq \frac{1}{3},\] which implies that $\sep(f) \leq \sqrt{3M(f)}$.

    Next, assume that $1 \leq \alpha \leq t$.  Here, we have \[\frac{M(f)}{\sep(f)^2} = \frac{\alpha(x^2 + y^2)}{(\alpha - x)^2 + y^2} = \frac{x^2 + y^2}{\alpha - 2x + \frac{x^2 + y^2}{\alpha}}.\]  Observe that the denominator, as a function of $\alpha$, will be maximized at either $\alpha = 1$ or $\alpha = t$.  In the case of $\alpha = 1$, we have already shown that \[\frac{x^2 + y^2}{(1 - x)^2 + y^2} \geq \frac{1}{3}.\]  So it remains to check that \[\frac{t(x^2 + y^2)}{(t-x)^2 + y^2} \geq \frac{1}{3}.\]  First, observe that since $t = x + y\sqrt{3}$, we have \[\frac{t(x^2 + y^2)}{(t-x)^2 + y^2} = \frac{t((t-y\sqrt{3})^2 + y^2)}{4y^2} = \frac{t^3 - 2\sqrt{3}t^2y+4ty^2}{4y^2}.\]  Next, observe that $(t^3 - 2\sqrt{3}t^2y+4ty^2)/(4y^2)$ is a nondecreasing function of $t$ since its partial derivative with respect to $t$ is $3(t - 2y/\sqrt{3})^2/(4y^2)$.  Note additionally that since $y \geq -x\sqrt{3}$, we have \[t = \sqrt{3}y + x \geq \left(\sqrt{3} - \frac{1}{\sqrt{3}}\right)y = \frac{2}{\sqrt{3}}y.\]  Since $(t^3 - 2\sqrt{3}t^2y+4ty^2)/(4y^2)$ is nondecreasing in $t$ and since $t \geq \frac{2}{\sqrt{3}}y$, we must have \[\frac{t^3 - 2\sqrt{3}t^2y+4ty^2}{4y^2} \geq \frac{2}{3\sqrt{3}}y.\]  But the fact that $x^2 + y^2 \geq 1$ together with the assumption that $y \geq -x\sqrt{3}$ implies that $y \geq \sqrt{3}/2$, and hence, we can finally conclude that \[\frac{M(f)}{\sep(f)^2} = \frac{t(x^2 + y^2)}{(t-x)^2 + y^2} = \frac{t^3 - 2\sqrt{3}t^2y+4ty^2}{4y^2} \geq \frac{2}{3\sqrt{3}}y \geq \frac{1}{3}.\]

    \noindent\underline{Case 2:} $x^2 + y^2 < 1$.

    Note now that since $x^2 + y^2 < 1$ and $y \geq -x\sqrt{3}$, we must have $x \geq -1/2$.

    Again, we have two subcases.  First, if $\alpha< 1$, then \[\frac{M(f)}{\sep(f)^2} = \frac{1}{(\alpha - x)^2 + y^2} \geq \frac{1}{(1 - x)^2 + y^2} = \frac{1}{1 - 2x + x^2 + y^2} \geq \frac{1}{2 - 2x} \geq \frac{1}{3}\] and we are done.

    Now assume that $1 \leq \alpha \leq t$.  The fact that $t \geq 1$ implies that $y = (t-x)/\sqrt{3} \geq 1/\sqrt{3}$.  In this case, we have \[\frac{M(f)}{\sep(f)^2} = \frac{\alpha}{(\alpha - x)^2 + y^2} = \frac{1}{\alpha - 2x + \frac{x^2 + y^2}{\alpha}},\] which is again minimized when $\alpha = 1$ or $\alpha = t$.  We have already shown that \[\frac{1}{(1 - x)^2 + y^2} \geq \frac{1}{3},\] so it  remains to show that \[ \frac{t}{4y^2}= \frac{t}{(t - x)^2 + y^2} \geq \frac{1}{3}.\]  
    If $y \leq \frac{\sqrt{3}}{2}$, then we have
    $\frac{t}{4y^2}\geq \frac{t}{3} \geq \frac{1}{3}$.  If $y \geq \sqrt{3}/2$, then we use the fact that $x \geq -\sqrt{1-y^2}$ to get
  \[\frac{t}{4y^2} = \frac{x + y\sqrt{3}}{4y^2} \geq \frac{y\sqrt{3}-\sqrt{1-y^2}}{4y^2}\geq \frac{1}{3}.\]

We can see that the constant $\sqrt{3}$ is optimal due to the polynomial $f(x) = x^3 - 1$.

\end{proof}

Finally, we turn our attention to the case where $f(x)$ is quartic with no real roots.

\begin{proof}[Proof of Proposition \ref{quarticSig02Prop}]
	Suppose that the roots of $f(x)$ are $\alpha,\bar{\alpha},\beta,\bar{\beta}$ where $\Im[\alpha],\Im[\beta] > 0$ and $|\alpha| = r \leq |\beta| = R$.  We first note that $\sep(f) = \min\{2\Im[\alpha],2\Im[\beta],|\alpha - \beta|\}).$ We have the following two cases:

	\noindent\underline{Case 1:} $2 r \leq R$

	In this case, we note that \[\sep(f) \leq 2\Im[\alpha] \leq 2r\leq 2r^{1/2}\left(\frac{R}{2}\right)^{1/2} = \sqrt{2}r^{1/2}R^{1/2} \leq \sqrt{2}M(f)^{1/4}.\]

	\noindent\underline{Case 2:} $r \leq R < 2r$

	In this case, we first observe that if $\Im[\alpha] < \frac{\sqrt{2}}{2}r^{1/2}R^{1/2}$ or if $\Im[\beta] < \frac{\sqrt{2}}{2}r^{1/2}R^{1/2}$, then we are done because \[\sep(f) \leq \min\{2\Im[\alpha],2\Im[\beta]\} \leq \sqrt{2}r^{1/2}R^{1/2} \leq \sqrt{2}M(f)^{1/4}.\]
	Hence, \[\alpha \in \left\{z \in \C : |z| = r \text{ and } \Im[z] \geq \frac{\sqrt{2}}{2}r^{1/2}R^{1/2}\right\}=:S\] and \[\beta \in \left\{z \in \C : |z| = R \text{ and } \Im[z] \geq \frac{\sqrt{2}}{2}r^{1/2}R^{1/2}\right\}=:T.\]

	As a result, we have 
	\begin{align*}
		\sep(f) 	&\leq |\alpha - \beta|
					\leq \sup_{\substack{z_1 \in S\\z_2 \in T}} |z_1 - z_2|\\
					&= \left|\left(\sqrt{r^2 - \frac{rR}{2}} + i\cdot \frac{\sqrt{2}}{2}r^{1/2}R^{1/2}\right) -  \left(-\sqrt{R^2 - \frac{rR}{2}} + i\cdot \frac{\sqrt{2}}{2}r^{1/2}R^{1/2}\right)\right|\\
					&= \sqrt{r^2 - \frac{rR}{2}} + \sqrt{R^2 - \frac{rR}{2}}.
	\end{align*} 
	We claim that \[\sqrt{r^2 - \frac{rR}{2}} + \sqrt{R^2 - \frac{rR}{2}} \leq \sqrt{2}r^{1/2}R^{1/2}.\]  To see this, divide both sides of the inequality by $r^{1/2}R^{1/2}$ to obtain the equivalent inequality \[\sqrt{\frac{r}{R} - \frac{1}{2}} + \sqrt{\frac{R}{r} - \frac{1}{2}} \leq \sqrt{2}.\]  Observe that this new inequality only depends on the ratio $x = \frac{R}{r}$, which we have bounded by $1 \leq x < 2.$  It is now a simple calculus problem to show that \[\sqrt{\frac{1}{x} - \frac{1}{2}} + \sqrt{x - \frac{1}{2}} \leq \sqrt{2}\] for all $1 \leq x < 2$ and the proof that $\sep(f) \leq \sqrt{2}M(f)^{1/4}$ is complete.

	To see that the bound is sharp, consider the family of polynomials \[f_t(x) = (x - t(1+i))(x - t(1-i))(x - t(-1+i))(x - t(-1-i))\] for $t \in \R$ with $t \geq 1/\sqrt{2}$.  Every polynomial in this family has $\sep(f_t) = \sqrt{2}M(f_t)^{1/4}.$
\end{proof}

%%%%%%%%%%%%%%%%%%%%%%%%%%%%%%%%%%%%%%%%%%%%%%%%%%%%%%%%%%%%%%%%%%%%%%%
%%%%%%%%%%%% ACKNOWLEDGMENTS %%%%%%%%%%%%%%%%%%%%%%%%%%%%%%%%%%%%%%%%%%
%%%%%%%%%%%%%%%%%%%%%%%%%%%%%%%%%%%%%%%%%%%%%%%%%%%%%%%%%%%%%%%%%%%%%%%
\section*{Acknowledgements}

The first author would like to thank Chris Sinclair and Joe Webster for their helpful suggestions regarding the development of this project. The second author thanks Gabriel Currier and Kenneth Moore for helpful discussions. Both authors would like to thank Andrej Dujella and Tomislav Pejkovi\'c for their comments on an early draft of this paper. The authors are also grateful to anonymous referees for their valuable comments and corrections. 

The first author would also like to gratefully acknowledge the financial support of a PIMS postdoctoral fellowship, along with NSERC grants RGPIN-2019-04844, RGPIN-2022-03559, and RGPIN-2018-03770. The research of the second author was supported in part by an NSERC fellowship.

\bibliographystyle{abbrv}

\bibliography{library}

\end{document}